\documentclass[5p]{elsarticle}

\usepackage{geometry}
\usepackage{amsfonts}
\usepackage{epsfig}
\usepackage{makeidx}
\usepackage{algorithm}
\usepackage{algorithmic}
\usepackage{amsmath,amssymb,latexsym}
\usepackage{enumerate}

\usepackage{algorithm}
\usepackage{algorithmic}
\usepackage{amsmath}
\usepackage{amssymb}
\usepackage{graphicx}
\usepackage{bbm}
\usepackage{rotating}
\usepackage{amsfonts}
\usepackage{graphicx}
\usepackage{psfrag}
\usepackage{color}
\usepackage{url}
\usepackage{ifthen}
\usepackage[T1]{fontenc}

\usepackage{pstricks}
\usepackage{pst-node}

\usepackage{latexsym}
\usepackage{xspace}
\usepackage{mathrsfs}
\usepackage{amscd}
\usepackage{amsxtra}
\usepackage{amsfonts}
\usepackage{verbatim}

\usepackage{epsfig, epstopdf}
\usepackage{pst-plot}

\newcommand{\PH}{\hbox{PH}}
\newcommand{\CH}{\hbox{$C^{free}_4$H}}

\newcommand{\NTH}{\hbox{NTH}}

\newtheorem{theorem}{Theorem}
\newtheorem{definition}[theorem]{Definition}

\newtheorem{lemma}[theorem]{Lemma}
\newtheorem{corollary}[theorem]{Corollary}

\newenvironment{proof}{\noindent\textit{Proof.}}{$\Box$}

\makeatletter
\newtheorem{myclaim}{Claim}
\newenvironment{proofclaim}{\noindent { \smallskip \it \underline{Proof}.} }{(\it \underline{End of the claim}.\smallskip)}

\graphicspath{{Pictures/}}

\makeatletter
\renewcommand{\subsection}{\@startsection{subsection}{2}%
  {\z@}%
  {.7\linespacing\@plus\linespacing}%
  {.5\linespacing}%
  {\normalfont\scshape\centering}}
\makeatother

\begin{document}

\begin{frontmatter}

\title{A fast algorithm to remove proper and homogenous pairs of cliques\\ (while preserving some graph invariants)}

\author[rvt]{Y. Faenza}
\ead{faenza@math.unipd.it}
\author[focal]{G. Oriolo\fnref{fn1}}
\ead{oriolo@disp.uniroma2.it}
\author[focal]{C. Snels}
\ead{snels@disp.uniroma2.it}

\fntext[fn1]{Corresponding address: Dipartimento di Informatica, Sistemi e Produzione, Facolt\`a di Ingegneria, Via del Politecnico 1, 00133 Rome, Italy.}

\address[rvt]{Dipartimento di Matematica Pura e Applicata, Universit\`a di Padova, Padua, Italy}
\address[focal]{Dipartimento di Informatica, Sistemi e Produzione, Universit\`a di Roma Tor Vergata, Rome, Italy}

\begin{abstract}
We introduce a family of reductions for removing proper and homogeneous pairs of cliques from a graph $G$. This family generalizes some routines presented in the literature, mostly in the context of claw-free graphs. These reductions can be embedded in a simple algorithm that in at most $|E(G)|$ steps builds a new graph $G'$ without proper and homogeneous pairs of cliques, and such that $G$ and $G'$ agree on the value of some relevant invariant (or property).
\end{abstract}


\end{frontmatter}

\section{Introduction}\label{intro}

A pair of vertex-disjoint cliques $\{K_1, K_2\}$ is \emph{homogeneous} if every vertex that is neither in $K_1$, nor in $K_2$ is either adjacent to all vertices from $K_1$, or non-adjacent to all of them, and similarly for $K_2$. Homogeneous pairs of cliques were first defined in the context of bull-free graphs \cite{CSb}, and seem to play a non-trivial role in combinatorial, structural and polyhedral properties of claw-free graphs. For instance, a well-known decomposition result by Chudnovsky and Seymour is as follows:

\begin{theorem}\label{thr:Ch-Seym-struct-cf-1} \cite{CS_Survey} For every connected claw-free graph G with $\alpha(G)\geq 4$, if $G$ does not admit a $1$-join and
there is no homogeneous pair of cliques in $G$, then either $G$ is a circular interval
graph, or $G$ is a composition of linear interval strips, XX-strips, and antihat strips.
\end{theorem}

See \cite{CS_Survey} for the definition of graphs and operations involved in Theorem \ref{thr:Ch-Seym-struct-cf-1}: we skip them, since they are of no use for the present paper. What is interesting to us is the fact that homogeneous pair of cliques are somehow an \emph{annoying} structure: as it is written in \cite{CS_Survey}, "There is also a ''fuzzy''
version of this (\emph{i.e. Theorem \ref{thr:Ch-Seym-struct-cf-1}}), without the hypothesis that there is no homogeneous pair of cliques
in G, but it is quite complicated". (This more complex version of the theorem is actually given in \cite{CS_global}.) A similar situation can be found in the structure theorem on Berge graphs \cite{CS_perfect}.

\smallskip
In the literature, some effort has been devoted to design {\em reduction techniques} to get rid of homogeneous pairs of cliques that are also \emph{proper}. We say that a pair of cliques $\{K_1, K_2\}$ is \emph{proper} if each vertex in $K_1$ is neither complete nor anticomplete to $K_2$, and each vertex in $K_2$ is neither complete nor anticomplete to $K_1$. Those reduction techniques are designed to preserve graph invariants, such as chromatic number \cite{kph, kr} and stability number \cite{OPS08bis}, or graph properties, such as the property of a graph of being quasi-line \cite{CO}, fuzzy circular interval \cite{OPS08}, or even facets of the stable set polytope \cite{eo}. The state of the art complexity for recognizing whether a graph $G(V,E)$ has some proper and homogeneous pairs of cliques is $O(|V(G)|^2|E(G)|)$ \cite{kr, Pi}.

\smallskip
In this paper, we introduce a reduction operation that generalizes and unifies those different techniques. It essentially replaces a proper and homogeneous pair of cliques $\{K_1, K_2\}$ with another pair of cliques $\{A_1, A_2\}$ that is homogeneous but non-proper. A large number of pairs $\{A_1, A_2\}$ can be used in our reduction, and the choice of a particular pair is done depending on some invariant (or property) we want the reduction to preserve. Regardless of this choice and of the number of proper and homogeneous clique of the input graph $G$, we show that our reduction can be embedded in a fast algorithm that iteratively replaces a proper and homogeneous pair of cliques $\{K^i_1, K^i_2\}$ with a non-proper and homogeneous one $\{A^i_1, A^i_2\}$, and outputs after $|E(G)|$ iterations a graph without proper and homogeneous pairs of cliques. We stress that the algorithm is not graph-class specific, i.e. it works with \emph{any} simple graph in input. Our main result will be then the following:

\begin{theorem}\label{thr:phelimin_preliminary}
Let $G(V,E)$ be a graph. Algorithm \ref{alg:phelimin} builds a sequence of graphs $G = G^0, G^1, \ldots, G^q$, with $q \leq |E(G)|$, such that $G^q$ has no proper and homogeneous pairs of cliques, and each $G^i$, $i < q$, is obtained from $G^{i-1}$ by replacing a proper and homogeneous pair of cliques $\{K^i_1, K^i_2\}$ with an homogeneous pair of cliques $\{A^i_1, A^i_2\}$. The algorithm can be implemented as to run in $O(|V(G)|^2|E(G)| + \sum_{i=1}^q p(i))$-time, if, for $i=1,\dots, q$, it takes $p(i)$-time to generate $G^{i+1}[A^i_1 \cup A^i_2]$, from the knowledge of $G^i, K_1^i$ and  $K_2^i$.
\end{theorem}

Combining this theorem with a few results from the literature, we will show some more facts, among which:
\begin{itemize}
\item we can reduce in time $O(|V(G)|^\frac{5}{2}|E(G)|)$ the coloring problem (resp. the maximum clique problem) on a graph $G(V,E)$ to the same problem on a graph $G'$ without proper and homogeneous pairs of cliques;
\item we can reduce in time $O(|V(G)|^2|E(G)|)$ the maximum weighted stable set problem on a graph $G(V,E)$ to the same problem on a graph $G'$ without proper and homogeneous pairs of cliques.
\end{itemize}

\section{Preliminaries}\label{sec:basics}
Given a simple graph $G(V,E)$, let $n=|V(G)|$ and $m=|E(G)|$. We denote by $uv$ an edge of $G$, while we denote by $\{u,v\}$ a pair of vertices $u,v\in V$. For a given $x \in V$, the \emph{neighborhood} $N(x)$ is the set of vertices $\{v\in V: xv\in E\}$. We say that $v$ is \emph{universal} to $u\in V$ if $v$ is adjacent to $u$ and to every vertex in $N(u) \setminus \{v\}$. Let $S\subset V$, then $x \notin S$ is \emph{complete} (resp. \emph{anticomplete}) to $S$ in $G$ if $S \cap N(x)=S$ (resp. $S \cap N(x)=\emptyset$).
Finally, we denote by $G[U]$ the subgraph induced on $G$ by $U\subseteq V$; a $C_4$ is an induced chordless cycle on four vertices.
\begin{definition}
\label{homo}
Let $G$ be a graph and  $\{K_1, K_2\}$ be a pair of non-empty and vertex-disjoint cliques. The pair $\{K_1, K_2\}$ is {\em homogeneous} if each vertex $z \not \in (K_1 \cup K_2)$ is either
complete or anti-complete to $K_1$ and either complete or anti-complete to $K_2$.
\end{definition}

\begin{definition}
\label{proper_vertex}
Let $K$ be a clique of a graph $G$ and let $v \notin K$.  $v$ is {\em proper} to $K$ if $v$ is neither complete nor anti-complete to $K$, and $P(K)$ is the set of vertices that are proper to $K$.
\end{definition}

\begin{definition}
\label{proper}
Let $G$ be a graph and  $\{K_1, K_2\}$ be a pair of non-empty and vertex-disjoint cliques. The pair $\{K_1, K_2\}$ is {\em proper} if each vertex $u \in K_1$ $(K_2$, respectively$)$ is proper to $K_2$ $(K_1)$. A pair of vertex-disjoint cliques that are proper and homogeneous is also called a {\em $\PH$ pair}.
\end{definition}

We skip the simple proof of the following lemma.

\begin{lemma}
\label{prop_univ} Let $G$ be a graph and  $\{K_1, K_2\}$ be a homogeneous pair of cliques. Then $\{K_1,K_2\}$ is proper if an only if, for each $i\in \{1, 2\}$ and $x\in K_i$, there exist $y_1, y_2\in K_i$ (possibly $y_1 = y_2$) such that $x$ is non-universal to $y_1$ and $y_2$ is non-universal to $x$.
\end{lemma}

In fact, one can show that for each clique $K_i$ of a proper pair  $\{K_1,K_2\}$ there always exist two vertices $x, y \in K_i$ that are non-universal to each other. Namely, we have the following (see Lemma $1$ in \cite{eo}):

\begin{lemma}
\label{lem:inducedC4} Let $\{K_1,K_2\}$ be a proper pair of cliques in a graph $G$. Then $G[K_1\cup K_2]$ contains $C_4$ as an induced subgraph.
\end{lemma}

Hence, when looking for a $\PH$ pair in a graph, one can start from a pair of vertices that are adjacent and not universal to each other, and then determine whether they have a {\em PH-embedding}, namely:

\begin{definition}
\label{embedding}
Let $u$ and $v$ be two adjacent vertices of a graph $G$. We say that $u$ and $v$ have a \emph{$\PH$-embedding} if they are not universal to each other, and there exists a $\PH$ pair of cliques $\{K_1,K_2\}$ such that $u,v\in K_1$. We also denote by $\PH(G)$ the set of pairs of vertices of $G$ that have a $\PH$-embedding.
\end{definition}

The next lemma is therefore trivial.

\begin{lemma}
\label{trivial}
If no pair of vertices of $G$ have a $\PH$-embedding, then $G$ has no $\PH$ pairs of cliques.
\end{lemma}

Given two adjacent vertices that are non-universal to each other, a simple algorithm recognizes in $O(n^2)$-time whether they have a $\PH$-embedding. This routine, which we report below, was independently proposed by King and Reed \cite{kr} and Pietropaoli \cite{Pi} (see also \cite{OPS08}). Actually King and Reed designed an algorithm for a slightly different problem: call $\{K_1,K_2\}$ a \emph{non-trivial homogeneous} ($\NTH$) pair of cliques in $G$ if $\{K_1,K_2\}$ is a homogeneous pair of cliques in $G$, and $G[K_1\cup K_2]$ has an induced $C_4$. Lemma \ref{lem:inducedC4} implies that each $\PH$ pair of cliques is a NTH pair of cliques, and one can immediately check that the converse does not always hold. But given a $\NTH$ pair of cliques $\{K_1,K_2\}$, one can obtain a $\PH$ pair of cliques $H_1,H_2$ with $H_1 \subseteq K_1$, $H_2 \subseteq K_2$, by iteratively removing from $\{K_1,K_2\}$ vertices that are non-proper to the opposite clique. Thus, in order to find a $\NTH$ pair one can look for a $\PH$ pair: this is exactly what King and Reed do in \cite{kr} (see Section $3$).

\begin{algorithm}
\caption{Finding a $\PH$-embedding} \label{PH_emb}
\begin{algorithmic} [1]
\REQUIRE A graph $G$, and a pair of adjacent vertices $\{u,v\}$ that are not universal to each other.
\ENSURE A $\PH$-embedding $\{K', K\}$ for $\{u, v\}$, if any.

\smallskip
\STATE  $K' := \{u, v\}$; $K := P(\{u,v\})$; \WHILE {$K$ is a
clique and $P(K) \neq K'$} \STATE{$K':= K$, $K:= P(K)$;} \ENDWHILE \STATE {\bf if} {$K$ is not a clique} {\bf then}
{there is no $\PH$-embedding for $\{u, v\}$: {\bf stop}.}
\STATE {\bf else} {$P(K) =K'$} and {$\{K,K'\}$ is a $\PH$-embedding for $\{u, v\}$: {\bf stop}.}
\end{algorithmic}
\end{algorithm}

\begin{theorem}\label{thr:PHemb}\cite{kr}, \cite{Pi} It is possible to implement Algorithm \ref{PH_emb} as to run in $O(|V(G)|^2)$.
\end{theorem}

Besides considering pairs of cliques that are proper and homogeneous, we will also consider pairs of cliques that are homogeneous but non-proper. This leads to the following definition:

\begin{definition}
\label{c_f-free}
Let $G$ be a graph and $\{A_1, A_2\}$ be a pair of non-empty and vertex-disjoint cliques that are not complete to each other. The pair $\{A_1, A_2\}$ is {\em $C^{free}_4$} if $G[A_1\cup A_2]$ has no induced $C_4$. A pair of cliques that is $C^{free}_4$ and homogeneous is also called a {\em $\CH$ pair}.
\end{definition}

It follows from Lemma \ref{lem:inducedC4} that no pair of $C^{free}_4$ cliques is proper. We skip the simple proof of the next lemma.
\begin{lemma}
\label{universal}
Let $G$ be a graph and $\{A_1, A_2\}$ be a pair of non-empty and vertex-disjoint cliques that are not complete to each other. Then $\{A_1, A_2\}$ is {\em $C^{free}_4$} if and only if the following holds: if $u$ and $v\in A_1$ then $u$ is universal to $v$ or $v$ is universal to $u$ (note that this property holds if and only if the same happens with the vertices of $A_2$).
 \end{lemma}

The next lemma analyzes the possible intersections between $\PH$ and $\CH$ pairs of cliques.

\begin{lemma}
\label{new_lemma}
Let $G(V,E)$ be a graph with a $\PH$ pair of cliques $\{K_1, K_2\}$ and a $\CH$ pair of cliques $\{A_1, A_2\}$.   Then $K_1 \cap A_2= K_2 \cap A_1= \emptyset$ or $K_1 \cap A_1= K_2 \cap A_2= \emptyset$.
\end{lemma}

\begin{proof}
We start with the following:
\begin{myclaim}\label{cl:1} $K_i\cap A_1 = \emptyset$ or $K_i\cap A_2 = \emptyset$, for $i=1,2$.
\end{myclaim}

\begin{proofclaim}
Without loss of generality, suppose to the contrary that there exist $a \in A_1$ and $b \in A_2$ such that $a,b \in K_1$. Being $K_1$ proper to $K_2$, there exist $c,d \in K_2$ (possibly non-distinct) such that $ad, bc \notin E$. We first show that $c, d \notin A_1\cup A_2$. Note that $d \notin A_1$ and $c \notin A_2$. Now suppose that $d\in A_2$; it follows that $d\neq c$. Since $c$ is adjacent to $d$ and not adjacent to $b$, and $\{A_1,A_2\}$ is a homogenous pair, it follows that $c\in A_1$. But then $a, b, c, d$ induce a $C_4$ on $G[A_1\cup A_2]$, and therefore neither $a$ is universal to $c$ nor $c$ is universal to $a$, which is a contradiction to Lemma \ref{universal}. We get an analogous contradiction if we assume that $c\in A_1$.

So $c,d \notin A_1\cup A_2$; being $ad, bc \notin E$ and $\{A_1,A_2\}$ a homogeneous pair, $c$ is anti-complete to $A_2$ and $d$ is anti-complete to $A_1$. Since $K_2$ is a clique, it follows that $K_2\cap (A_1\cup A_2)=\emptyset$. Since $A_1\cup A_2$ is not a clique, there exist $a' \in A_1$, $b' \in A_2$ such that $a'b' \notin E$. Note that $da'\notin E$ and that $a'\notin K_2$. We now show that $a'\notin K_1$. For, suppose the contrary; then $b'\neq b$ and $b'\notin K_1$, and so $b'$ is proper to $K_1$ and therefore belongs to $K_2$, which is a contradiction, since we already argued that $K_2\cap (A_1\cup A_2)=\emptyset$.

Hence $a'\notin K_1\cup K_2$. Since $\{K_1, K_2\}$ is a proper pair, there exists a vertex
$e \in K_2$ such that $ea \in E$. Since $K_2\cap (A_1\cup A_2)=\emptyset$ and $\{A_1,A_2\}$ is a homogeneous pair, it follows that $ea' \in E$. On the other hand, we observed that $da'\notin E$. But then $a'$ is proper to $K_2$, contradicting $a' \notin K_1$.
\end{proofclaim}

\smallskip
From the claim, we may assume without loss of generality that $K_1 \cap A_1= \emptyset$. In this case, the statement follows if $K_2 \cap A_2= \emptyset$, so suppose that there exists $v_2\in K_2 \cap A_2$. It again follows from the previous claim that $K_2 \cap A_1 \neq \emptyset$; hence the statement follows if $K_1 \cap A_2= \emptyset$. So suppose that there exists $v_1\in K_1 \cap A_2$; since $\{K_1, K_2\}$ is a proper pair, it follows that $v_1, v_2\in A_2$ are not universal to each other, a contradiction to Lemma \ref{universal}.
\end{proof}

\section{An algorithm for removing proper and homogeneous pairs}\label{sec:reduction}

We now define an operation of {\em reduction} that is crucial for the paper. This operation essentially replaces a $\PH$ pair of cliques with a $\CH$ pair of cliques. The latter pair will be defined through a suitable graph that we call, for shortness, a {\em non-proper 2-clique}.

\begin{definition}
\label{gadget}
A {\em non-proper 2-clique} $H_{\{A_1,A_2\}}$ is a graph with a $C^{free}_4$ pair of cliques $\{A_1,A_2\}$, such that $V(H_{\{A_1,A_2\}}) = A_1\cup A_2$.
\end{definition}

\begin{definition}
\label{reduction}
Let $G$ be a graph with a $\PH$ pair of cliques $\{K_1,K_2\}$. Also let $H_{\{A_1,A_2\}}$ be a non-proper 2-clique graph vertex-disjoint from $G$. The {\em $\PH$ reduction} of $G$ with respect to $(K_1, K_2, H_{\{A_1,A_2\}})$ returns a new graph $G|_{K_1, K_2, H_{\{A_1,A_2\}}}$ defined as follows:

\begin{itemize}
\item $V(G|_{K_1, K_2, H_{\{A_1,A_2\}}})=(V(G)\setminus (K_1\cup K_2)) \cup (A_1 \cup A_2)$;
\item Let $x, y$ be vertices of $G|_{K_1, K_2, H_{\{A_1,A_2\}}}$. The edge $xy \in E(G|_{K_1, K_2, H_{\{A_1,A_2\}}})$ if and only if one of the following holds:
\begin{itemize}
\item $xy \in E(G)$ with $x,y \notin K_1\cup K_2$;
\item $xy \in E(H_{\{A_1,A_2\}})$ with $x,y \in A_1\cup A_2$;
\item $y \in A_1$, $x \notin K_1 \cup K_2$ and $x$ is complete to $K_1$;
\item $y \in A_2$,  $x \notin K_1 \cup K_2$ and $x$ is complete to $K_2$.
\end{itemize}
\end{itemize}
\end{definition}

We skip the trivial proof of the following lemma.

\begin{lemma}\label{lem:chain}
The graph $G|_{K_1, K_2, H_{\{A_1,A_2\}}}$ is such that the following properties hold:
\begin{itemize}
\item $\{A_1, A_2\}$ is a $\CH$ pair of cliques;
\item if $x,y\in A_1$ (resp. $x,y\in A_2$), then $x$ is universal to $y$ or $y$ is universal to $x$;
\item if $|K_1|\geq |A_1|$ and $|K_2|\geq |A_2|$, then the graph $G|_{K_1, K_2, H_{\{A_1,A_2\}}}$ can be built in time $O(|V(G)|^2)$ and $|V(G|_{K_1, K_2, H_{\{A_1,A_2\}}})|\leq |V(G)|$.
\end{itemize}
\end{lemma}

The following crucial lemma shows that all the $\PH$ pairs of $G|_{K_1, K_2, H_{\{A_1,A_2\}}}$  are ``inherited" by the input graph $G$.

\begin{lemma}\label{lem:C4}
Let $\{w_1,w_2\}$ be a pair of adjacent vertices of $G|_{K_1, K_2, H_{\{A_1,A_2\}}}$ with a $\PH$-embedding. Then:
\begin{enumerate}
\item $w_1$ and $w_2$ do not both belong to $A_1\cup A_2$;
\item if $w_1,w_2 \notin A_1\cup A_2$, then $\{w_1,w_2\}$ also admits a $\PH$-embedding in $G$;
\item if $w_1 \in A_1$ (resp. $w_1\in A_2$) and $w_2\notin A_1\cup A_2$, then, for each $a \in K_1$ (resp. $a\in K_2$), $\{a, w_2\}$ admits a $\PH$-embedding in $G$.
\end{enumerate}
\end{lemma}

\begin{proof} Throughout the proof, when referring to vertices of $G|_{K_1, K_2, H_{\{A_1,A_2\}}}$, we call \emph{artificial} the vertices of $A_1\cup A_2$, and \emph{non-artificial} the others. Moreover, we let $G' = G|_{K_1, K_2, H_{\{A_1,A_2\}}}$ and let $\{K'_1, K'_2\}$ be a $\PH$-embedding for $\{w_1,w_2\}$ in $G'$.

It follows from Lemma \ref{lem:chain} that $\{A_1, A_2\}$ is a $\CH$ pair of cliques of $G'$. Therefore it follows from Lemma \ref{new_lemma} that $K'_1 \cap A_2= K'_2 \cap A_1= \emptyset$ or $K'_1 \cap A_1= K'_2 \cap A_2= \emptyset$. Now suppose that $w_1, w_2 \in A_1\cup A_2$, and recall that, by definition, $w_1, w_2 \in K'_1$. It follows that either $w_1, w_2 \in A_1$, or $w_1, w_2 \in A_2$. Thus, there exist two vertices of $A_1$ (resp. $A_2$) that are non-universal to each other, contradicting Lemma \ref{lem:chain}. Therefore $w_1$ and $w_2$ do not both belong to $A_1\cup A_2$, i.e. statement $1$ holds.

W.l.o.g. in the following we assume that $K_1' \cap A_2= K_2' \cap A_1= \emptyset$. Now define the sets $H_1,H_2$ of
vertices in $G$ as follows: for $i=1,2$, if $K_i'$ has no
artificial vertices, define $H_i=K_i'$; otherwise $H_i=(K_i'\cap
V(G))\cup K_i$. Note that this implies that $H_1\cap K_2 = H_2 \cap
K_1= \emptyset$ and that $H_1$ and $H_2$ are cliques.

\begin{myclaim}\label{cl:3} Let $u, v\in K_1'$ (respectively $K_2'$) be two non-artificial vertices of $G'$ such that $u$ is non-universal to $v$ in $G'$. Then $u, v\in H_1$ (respectively $H_2$) and $u$ is non-universal to $v$ in $G$.
\end{myclaim}

\begin{proofclaim} We prove the statement for $u,v \in K_1'$. Since $u,v$ are non-artificial,
$u,v \in H_1$ by definition. By hypothesis, there exists $z \in
K_2'$ s.t. $uz \notin E(G')$, $vz \in E(G')$. If $z$ is non-artificial, $z
\in H_2$ by definition, thus $u$ is non-universal to $v$ in $G$.
Suppose now $z$ is artificial, then $z\in A_2$, since $K_2' \cap A_1= \emptyset$. Then by construction $v$ is
complete and $u$ anticomplete to $K_2$ in $G$, thus $u$ is
non-universal to $v$
in $G$. \end{proofclaim}

\begin{myclaim}\label{cl:4} Let $u,v \in K_1'$ (respectively
$K_2'$), and suppose $u$ is artificial and $v$ is not. Then $\{v\}\cup K_1 \subseteq H_1$ (resp. $\{v\}\cup K_2 \subseteq H_2$).
Furthermore:
\begin{enumerate}
\item If $u$ is non-universal to $v$, then $a$ is non-universal to
$v$ for each $a \in K_1$ (respectively $K_2$).\item  If $v$ is
non-universal to $u$, then $v$ is non-universal to $a$, for each
$a \in K_1$ (resp. $K_2$).
\end{enumerate}
\end{myclaim}

\begin{proofclaim} We prove the statement for $u,v \in K_1'$. We are assuming that $K_1' \cap A_2= \emptyset$, hence $u\in A_1$. So by definition, $\{v\}\cup K_1 \subseteq H_1$. Suppose $u$ is non-universal to $v$: there exists $z \in K_2'$
s.t. $uz \notin E(G')$, $vz \in E(G')$. If $z$ is an artificial vertex,
then $z \in A_2$, which implies that $v$ is complete to
$K_2$, while each vertex $a\in K_1$ is proper to $K_2$. If $z$ is non-artificial, then by construction
$z$ is anticomplete to $K_1$ while $vz \in E(G)$. This shows 1. Now suppose that $v$ is non-universal to $u$, i.e. there exists $z
\in K_2'$ such that $uz \in E(G'), vz \notin E(G')$. If $z$ is an
artificial vertex, then $K_2 \subseteq H_2$ and $v$
is anticomplete to $K_2$; since each vertex $a\in K_1$ is proper to $K_2$, $v$ is
non-universal to $a$. If $z$ is non-artificial, then $z$ is
complete to $K_1$ in $G$, while $zv \notin E(G)$; thus, $v$ is
non-universal to $a \in K_1$.\end{proofclaim}

\begin{myclaim}\label{cl:5} $\{H_1, H_2\}$ is a $\PH$ pair of cliques in $G$.
\end{myclaim}

\begin{proofclaim}
We already observed that $H_1$ and $H_2$ are cliques, and it is straightforward to see that $\{H_1,H_2\}$ is a homogeneous pair.  So we conclude the proof by showing that $H_1$
is proper to $H_2$ (the other case following by symmetry).

We need to show that each vertex $x\in H_1$ has at least one neighbor and at least one non-neighbor in $H_2$.
Recall that $x \notin K_2$. Suppose first that $x \in K_1$; then by construction $K_1 \subseteq
H_1$ and $K_1'$ has at least one artificial vertex, say $a$. Since
$\{K_1',K_2'\}$ is a proper pair, it follows from Lemma \ref{prop_univ} that there exist a vertex $t_1\in K_1'$ to which
$a$ is non-universal, and a vertex $t_2\in K_1'$ which is non-universal to $a$. If $t_1$ or $t_2$ is artificial, then $K_2'$ intersects $A_2$ (recall that $a,t_1, t_2 \in A_1$ have the same neighborhood outside $K_2'$) and consequently, by construction, $K_2 \subseteq H_2$; then the statement follows since $\{K_1,K_2\}$  is a proper pair of cliques.
Conversely, if both $t_1$ and $t_2$ are non-artificial, then, using Claim \ref{cl:4}, we conclude
that in $G$ $x$ is non-universal to $t_1$ and that $t_2$ is non-universal to $x$, and therefore $x$ has at least one neighbor and at least one non-neighbor in $H_2$.

Suppose now $x \notin K_1$: then, $x$ is a non-artificial vertex of $K_1'$, and since
$\{K_1',K_2'\}$ is proper, it follows again from Lemma \ref{prop_univ} that there exist a vertex $t_1\in K_1'$ to which
$x$ is non-universal, and a vertex $t_2\in K_1'$ which is non-universal to $x$. If both $t_1$ and $t_2$ are non-artificial, then also in $G$ we have that $x$ is non-universal to $t_1$ and $t_2$ is non-universal to $x$. If $t_1$ or $t_2$ is artificial, then thanks to Claim \ref{cl:4}, we may suitably replace $t_1$ or $t_2$ with vertices from $K_1$ as to get the same conclusion.
\end{proofclaim}

We conclude the proof of the lemma: part $2$ holds by Claims \ref{cl:3} and \ref{cl:5}, while part $3$ holds by Claims \ref{cl:4} and \ref{cl:5}.
\end{proof}

\medskip
As we show in the following, if we iterate the reduction of Definition \ref{reduction}, we end up, in at most $|E(G)|$ steps, with a graph without $\PH$ pairs of cliques.
We first need a definition and a simple lemma, going along the same lines of Definition \ref{reduction} and Lemma \ref{lem:C4}. For a graph $G$, we denote by ${V(G)}\choose{2}$ the set of unordered pairs of vertices of $V(G)$.

\begin{definition}
\label{reduction_bis}
Let $G$ and $G' := G|_{K_1, K_2, H_{\{A_1,A_2\}}}$ be as in Definition \ref{reduction}, and let $S\subseteq {{V(G)}\choose{2}}$. The set $S|_{K_1, K_2, H_{\{A_1,A_2\}}}\subseteq {{V(G')}\choose{2}}$ is the set of pairs $\{x, y\}$ such that one of the following hold:
\begin{itemize}
\item $\{x, y\}\in S$ and $x, y \notin A_1\cup A_2$;
\item $x \in A_1$, $y\notin A_1\cup A_2$ and $\{\{a, y\} \mid a \in K_1\} \subseteq S$;
\item $y \in A_2$, $x\notin A_1\cup A_2$ such that $\{\{x, a\} \mid a \in K_2\} \subseteq S$.
\end{itemize}
\end{definition}

\begin{corollary}
\label{newcorollary}
Let $G$, $G':= G|_{K_1, K_2, H_{\{A_1,A_2\}}}$, $S$ and $S':=  S|_{K_1, K_2, H_{\{A_1,A_2\}}}$ be as in Definition \ref{reduction} and  Definition \ref{reduction_bis}.
\begin{itemize}
\item[$(i)$] If $S$ is a superset of $\PH(G)$, then $S'$ is a superset of $\PH(G')$.
\item[$(ii)$] If $|K_1|\geq |A_1|$ and $|K_2|\geq |A_2|$, then $|S'| < |S|$ and $S'$ can be built from $S$ in time $O(|V(G)|^2)$.
\end{itemize}
\end{corollary}

\begin{proof}
$(i)$ Pick any pair $\{w_1,w_2\}$ of vertices of $G'$ which admit a $\PH$-embedding in $G'$: by part $(1)$ of Lemma \ref{lem:C4}, they cannot both belong to $A_1\cup A_2$. Suppose that $w_1, w_2 \notin A_1\cup A_2$. Then, by part $(2)$ of Lemma \ref{lem:C4}, $\{w_1,w_2\}$ also have a $\PH$-embedding in $G$ and thus $\{w_1,w_2\}\in S$. Then, by construction, $\{w_1, w_2\} \in S'$. Now, suppose that exactly one of them belongs to $A_1\cup A_2$, w.l.o.g. $w_1$, and let first $w_1\in A_1$; then by part $(3)$ of Lemma \ref{lem:C4}, for each $a \in K_1$, $\{a,w_2\}$ is a pair of vertices with a $\PH$-embedding in $G$, i.e. $\{\{a,w_2\}, a \in K_1\}\subseteq PH(G)\subseteq S$. Then, by construction, $\{w_1,w_2\} \in S'$. A similar argument works for $w_1 \in A_2$. $(ii)$ The statements holds easily by construction.
\end{proof}

\medskip
We are now ready to give our algorithm, see Algorithm \ref{alg:phelimin} in the following. Note that it is fully determined, but for the choice of the non-proper 2-clique graph $H_{\{A^i_1,A^i_2\}}$ to be used in each iteration $i$. In fact, the definition of $H_{\{A^i_1,A^i_2\}}$ will in general depend on $G^i, K_1^i$ and  $K_2^i$: this will be discussed in the next section. Given our previous arguments, it is easy to conclude that Theorem \ref{thr:phelimin_preliminary} correctly predicts the output and the time complexity of Algorithm \ref{alg:phelimin}: we skip details.

Let us remark here that in Algorithm \ref{alg:phelimin} we start with a set $S^0=E(G)$, since we assumed no prior knowledge is available on the pair of vertices of $G$ that are candidate to have a $\PH$-embedding. For specific graphs we may have a better knowledge of those, and consequently start from a set $S^0$ smaller in size. This may lead to asymptotically faster implementations of Algorithm \ref{alg:phelimin}.

\begin{algorithm}
\caption{Eliminating all proper and homogeneous pairs of cliques}
\label{alg:phelimin}
\begin{algorithmic} [1]
\REQUIRE A graph $G$.
\ENSURE A graph $G^q$, without $\PH$ pairs of cliques, that is obtained from $G$ by successive $\PH$ reductions.

\smallskip
\STATE  $i:= 0$; $G^0 := G$; $S^0:= E(G)$;
\WHILE {$S^i$ is non-empty}
	\STATE{pick a pair $\{u, v\}\in S^i$;}
	\STATE{using Algorithm \ref{PH_emb} check whether the pair $\{u, v\}\in S^i$ has a $\PH$-embedding in $G^i$;}
	\IF{$u,v$ have a $\PH$-embedding $\{K^i_1, K^i_2\}$}
			\STATE {let $H_{\{A^i_1,A^i_2\}}$ be a non-proper 2-clique graph vertex-disjoint from $V(G^0)\cup V(G^1) \cup \ldots \cup V(G^{i})$ and such that $|K^i_1|\geq |A^i_1|$ and $|K^i_2|\geq |A^i_2|$;}
			\STATE {$G^{i+1}:= G^i|_{K^i_1, K^i_2, H_{\{A^i_1,A^i_2\}}}$ (see Definition \ref{reduction});}
			\STATE {$S^{i+1}:= S^i|_{K^i_1, K^i_2, H_{\{A^i_1,A^i_2\}}}$  (see Definition \ref{reduction_bis});}
            \STATE {$i:= i+1$;}
	\ELSE
\STATE {remove the pair $\{u,v\}$ from $S^i$;}
\ENDIF
\ENDWHILE
\STATE  $q:= i$.
\RETURN $G^q$.
\end{algorithmic}
\end{algorithm}

\section{Preserving some graph invariant or property}\label{sec:application}

In this section, we show that suitable $\PH$ reductions preserve graph invariants, such as chromatic number, stability number, and clique number, or graph properties, such as perfection, or the property of a graph of being fuzzy circular interval. Most of these reductions were in fact proposed in the literature in specific contexts, but they can actually be embedded in the unifying setting of $\PH$ reductions.

In some cases \cite{eo, kph, kr, OPS08} the reductions that were used have the following form: take a $\PH$ pair of cliques $\{K_1, K_2\}$ and remove some suitable set of edges between vertices of $K_1$ and vertices of $K_2$ so that, in particular, in the resulting graph,
no $C_4$ is contained in the subgraph induced by $K_1\cup K_2$. In another case \cite{OPS08bis} the reduction has the following form: take a $\PH$ pair of cliques $\{K_1, K_2\}$ and add all possible edges between vertices of $K_1$ and vertices of $K_2$ but one. It is easy to show that all those types of reductions can be interpreted in terms of our $\PH$ reduction, so we skip such details when presenting them. Therefore, they can be embedded into the iterative framework of Algorithm \ref{alg:phelimin}, and one may rely on the complexity bound given by Theorem \ref{thr:phelimin_preliminary}.

\smallskip
We begin with a reduction introduced by King and Reed \cite{kph, kr} for removing edges in a $\PH$ pair of cliques while preserving the chromatic number. Recall that $\chi(G)$ denotes the chromatic, $\chi_f(G)$ the fractional, and $\omega(G)$ the clique number of a graph $G$.

\begin{lemma}\label{lem:coloring}\cite{kph}
Let $G$ be a graph and suppose that we are given a $\PH$ pair of cliques $\{K_1, K_2\}$ of $G$. Also, let $X$ be a maximum clique in $G[K_1 \cup K_2]$, and let $G'$ be the graph obtained from $G$ by removing each edge $uv\in E(G)$ such that: $u\in K_1$; $v\in K_2$; $\{u, v\} \not\subseteq X$.  Then:
\begin{itemize}
\item[(i)] $G'$ can be built in time $O(|V(G)|^\frac{5}{2})$ (from the knowledge of $G$, $K_1$ and $K_2$);
\item[(ii)] $\chi(G)=\chi(G')$, $\chi_f(G)=\chi_f(G')$ and each $k$-coloring of $G'$ can be extended into a $k$-coloring of $G$ of in time $O(|V(G)|^\frac{5}{2})$.
\item[(iii)] $\omega(G)=\omega(G')$, and each clique of $G'$ is also a clique of $G$.
\item[(iv)] If $G$ is claw-free (resp. quasi-line; perfect), then $G'$ is claw-free (resp. quasi-line; perfect).
\end{itemize}
\end{lemma}

(One should mention that Lemma \ref{lem:coloring} can be extended to the case where $\{K_1, K_2\}$ is a {\em nonskeletal} and homogeneous pair of cliques \cite{kph}. Also, Andrew King \cite{kpe} pointed us that this lemma is non-trivially implied by some proofs in \cite{CO}. In that paper, Chudnovsky and Ovetsky introduce another reduction for $\PH$ pairs of cliques, which is quite similar to the one above. This reduction preserves quasi-liness, while not increasing the clique number of $G$. It is a simple exercise to show that the reduction in \cite{CO} can be interpreted in terms of our $\PH$ reduction. Finally, we mention that proposition $(iii)$ of Lemma \ref{lem:coloring} is not stated in \cite{kph}, but it is almost straightforward.)

\smallskip

By embedding the reduction above in the iterative framework of Algorithm \ref{alg:phelimin}, we can reduce the problem of computing the chromatic (resp. clique) number on a given graph $G$ to the same problem on a graph $G'$ without $\PH$ pairs of cliques. 

\begin{corollary}\label{cor:iterat-clique-col}
From a graph $G$ one can obtain in time $O(|V(G)|^\frac{5}{2}|E(G)|)$ a graph $G'$ without $\PH$ pairs of cliques such that $\chi(G)=\chi(G')$ and $\omega(G)=\omega(G')$. One can also derive an optimal coloring of $G$ from an optimal coloring in $G'$ in time $O(|V(G)|^\frac{5}{2}|E(G)|)$, while a maximum clique in $G'$ is also a maximum clique in $G$.
\end{corollary}

As argued by Li and Zang \cite{LZ}, the maximum weighted clique problem in the complement of a bipartite graph can be reduced to maximum flow, and hence solved in time $O(n^3)$. By building on the latter fact (and slightly increasing the complexity), Corollary \ref{cor:iterat-clique-col} can be extended to the computation of a graph $G'$ without $\PH$ cliques that preserves the maximum weighted clique and its value.

\smallskip
Consider now the maximum weighted stable set problem. Oriolo, Pietropaoli, and Stauffer \cite{OPS08bis} provide a reduction that preserves the value of a maximum weighted stable set. (We refer to \cite{OPS08bis} for more details and for the precise definition of the reduction, which is actually stated for the more general class of {\em semi-homogeneous} pairs of cliques.) By embedding their reduction in Algorithm \ref{alg:phelimin}, we obtain the following lemma:

\begin{corollary}\label{cor:weighreduction}
Let $G(V,E)$ be a graph with a weight function $w: V\mapsto{\mathbb R}$ defined on its vertices. In time $O(|V(G)|^2|E(G)|)$ one can build a graph $G'$ without $\PH$ pairs of cliques such that a maximum weighted stable set of $G'$ is also a maximum weighted stable set of $G$.
\end{corollary}

Interestingly, if we now move from the maximum weighted stable set problem to the stable set polytope $STAB(G)$ of a graph $G$, we can also embed a result in \cite{eo} in our framework. Eisenbrand et al. show -- see the remark following Lemma 5 in \cite{eo} -- that each facet of the stable set polytope $STAB(G)$ is also a facet of another graph $G'$ (obtained from $G$ by removing edges) that does not contain any $\PH$ pair of cliques. As one easily checks (cfr. the proof of Lemma 5 in \cite{eo}), also their result can be phrased in the framework of Algorithm \ref{alg:phelimin}.

\medskip
We now move from graph invariants to graph properties. First, Oriolo, Pietropaoli, and Stauffer \cite{OPS08} show that a suitable reduction of $\PH$ pairs of cliques preserves the property of a graph of being, or not being, a fuzzy circular interval graph, and they exploit this fact in an algorithm for recognizing fuzzy circular interval graphs. Their reduction can also be embedded in our framework. In fact, Theorem \ref{thr:phelimin_preliminary} is already used in \cite{OPS08} for bounding the complexity of the recognition algorithm. Finally, every $\PH$ reduction preserves perfection, and under very general conditions it does not turn a non-perfect graph into a perfect one. We give just a sketch of the proof of the latter fact, since the arguments used are quite standard.

\begin{lemma}\label{lem:perf}
Let $G$ be a perfect graph with a $\PH$ pair of cliques $\{K_1, K_2\}$. Also let $H_{\{A_1,A_2\}}$ be a non-proper 2-clique graph vertex-disjoint from $G$. Then the graph $G|_{K_1, K_2, H_{\{A_1,A_2\}}}$ is perfect. The converse implication holds true if $A_1$ is not anticomplete to $A_2$.
\end{lemma}

\begin{proof}
Recall that a graph is perfect if and only if it contains neither long odd holes, nor long odd anti-holes, {\em long} meaning of length at least $5$~\cite{CS_perfect}. Let $\{Q_1,Q_2\}$ be a homogeneous pair of cliques in a graph $G$: it is easy to show that each long odd-hole (resp. each long odd anti-hole) of $G$ takes at most one vertex from $Q_1$ and at most one vertex from $Q_2$.

Suppose first that $G'=G|_{K_1, K_2, H_{\{A_1,A_2\}}}$ is not perfect, i.e. there is an induced subgraph $H'$ of $G'$ that is either a long odd-hole or a long anti-hole. By building on the fact that $|V(H')\cap A_1|\leq 1$ and $|V(H')\cap A_2|\leq 1$, one can easily construct an odd-hole (resp. an odd anti-hole) of $G$ from $H'$, thus showing that $G$ is not perfect as well. Let now $A_1$ be not anticomplete to $A_2$ in $G'$; then, one can analogously show that if $G$ is not perfect, neither is $G'$.
\end{proof}

\medskip

We conclude by pointing out that, with the exception of the reduction from Lemma \ref{lem:coloring} (since $X \subseteq K_1$ or $X \subseteq K_2$ may happen), all the reductions from the current section do not turn an imperfect graph into a perfect one.

\bigskip\noindent
{\bf Note.} While preparing this paper, we became aware that M. Chudnovsky and A. King independently found a result similar to Theorem \ref{thr:phelimin_preliminary}, even though they reduce all proper and homogeneous pairs of cliques at once \cite{CK10}.

\section*{Acknowledgments} We thank Andrew King for reading a previous version of this paper and the anonymous referee for his/her comments. Yuri Faenza's research was supported by the Progetto di Eccellenza 2008-2009 of the Fondazione Cassa di Risparmio di Padova e Rovigo.


\begin{thebibliography}{5}

\footnotesize \baselineskip = 10pt

\bibitem{CK10}
M. Chudnovsky, and A. King.
\newblock \emph{Personal Communication}, 2010.

\bibitem{CO}
M. Chudnovsky, and A. Ovetsky. Coloring quasi-line graphs. \emph{J. Graph Theory}, Vol. 54, 41-50, 2007.

\bibitem{CS_perfect}
M. Chudnovsky, N. Robertson, P. Seymour, and R. Thomas. The Strong Perfect Graph Theorem. \emph{Ann. Math.}, Vol.164, 51-229, 2006.

\bibitem{CS_Survey}
M. Chudnovsky, and P. Seymour. The Structure of Claw-free
Graphs. {\em Surveys in Combinatorics 2005,  London Math. Soc.
Lecture Note Series}, Vol. 327, 153-171, 2005.

\bibitem{CS_global}
M. Chudnovsky, and P. Seymour. Claw free Graphs V. Global structure. {\em J. Comb. Theory, Ser. B}, Vol. 98, 1373-1410, 2008.

\bibitem{CSb}
V. Chv\'atal, and N. Sbihi. Bull-free Berge graphs are perfect. {\em Graphs and Combinatorics}, Vol. 3, 127-139, 1987.

\bibitem{eo}
F. Eisenbrand, G. Oriolo, G. Stauffer, and P. Ventura. Circular One
Matrices and the Stable Set Polytope of Quasi-Line Graphs. \emph{Combinatorica},
Vol. 28(1), 45-67, 2008.

\bibitem{kph}
A.D. King. Claw-free graphs and two conjectures on $\omega, \Delta$ and $\chi$. PhD Thesis, Mc Gill University, Montreal, 2009.

\bibitem{kpe}
A. King.
\newblock \emph{Personal Communication}, 2010.

\bibitem{kr}
A.D. King, and B.A. Reed. Bounding $\chi$ in Terms of $\omega$
and $\Delta$ for Quasi-Line Graphs. \emph{J. Graph Theory},
Vol. 59, 215-228, 2008.

\bibitem{LZ}
X. Li, and W. Zang. A Combinatorial Algorithm for Minimum Weighted
Colorings of Claw-Free Perfect Graphs. \emph{J. Comb. Opt.}, Vol. 9, 331-347, 2005.

\bibitem{OPS08bis}
G. Oriolo, U. Pietropaoli, and G. Stauffer. A new algorithm for the maximum weighted stable set problem in claw-free graphs. In A. Lodi, A. Panconesi and G. Rinaldi, editors, {\em Proceedings Thirteenth IPCO Conference}, 77-96, 2008.

\bibitem{OPS08}
G. Oriolo, U. Pietropaoli, and G. Stauffer. On the Recognition of Fuzzy Circular Interval Graphs. {\em Submitted Manuscript}, 2010.

 \bibitem{Pi}
U. Pietropaoli. Some classical combinatorial problems on circulant and claw-free graphs. PhD Thesis, Universit\`a di
Roma La Sapienza, 2008. An extended abstract appeared in {\em 4OR}, Vol. 7(3),
297-300, 2009.

\end{thebibliography}
\end{document}